\documentclass[11pt]{amsart}
\usepackage[left=1in,right=1in,top=1in,bottom=1in]{geometry}

\usepackage[all,cmtip]{xy}
\usepackage{amssymb}
\usepackage[english]{babel}
\usepackage{mathrsfs}

\newtheorem{theorem}{Theorem}[section]
\newtheorem{lemma}[theorem]{Lemma}

\newtheorem{corollary}[theorem]{Corollary}

\theoremstyle{definition}
\newtheorem{definition}[theorem]{Definition}

\newtheorem{question}[theorem]{Question}

\newtheorem{claim}{Claim}

\def\N{\mathbb{N}}

\def\Z{\mathbb{Z}}

\def\Z{\mathbb{Z}}
\def\Q{\mathbb{Q}}

\def\F{\mathbb{F}}

\def \Zet[#1]{\lceil #1 \rceil}

\def\supp{\mathrm{supp}}

\def\structure#1{\langle{#1}\rangle}

\def\cont{\mathfrak{c}}

\begin{document}
\title[Selectively pseudocompact groups without convergent sequences]
{Selectively pseudocompact groups without non-trivial convergent sequences}

\author[D. Shakhmatov]{Dmitri Shakhmatov}
\address{Division of Mathematics, Physics and Earth Sciences\\
Graduate School of Science and Engineering\\
Ehime University, Matsuyama 790-8577, Japan}
\email{dmitri.shakhmatov@ehime-u.ac.jp}
\thanks{The first listed author was partially supported by the Grant-in-Aid for Scientific Research~(C) No.~26400091 of the Japan Society for the Promotion of Science (JSPS)}

\author[V. Ya\~nez]{V\'{\i}ctor Hugo Ya\~nez}
\address{Master's Course, Graduate School of Science and Engineering\\
Ehime University, 
Matsuyama 790-8577, Japan}
\email{victor\textunderscore yanez@comunidad.unam.mx}
\thanks{This paper was written as part of the second listed author's Master's Program at the Graduate School of Science and Engineering of Ehime University. The second listed author was partially supported by the 2016 fiscal year grant of the Matsuyama Saibikai.}

\subjclass[2010]{Primary: 22A05; Secondary: 54A20, 54D30, 54H11}

\keywords{pseudocompact space, strong pseudocompactness,
$p$-compactness,
selective sequential pseudocompactness, non-trivial convergent sequence, free precompact Boolean group}

\begin{abstract}
The existence of 
a countably compact group without non-trivial convergent sequences in ZFC alone is a major open problem in 
topological group theory.
We give a ZFC example of a Boolean topological group $G$ without non-trivial convergent sequences having the following ``selective'' compactness property: For each free ultrafilter
$p$ on $\N$ and every sequence $\{U_n:n\in\N\}$ of non-empty open subsets of $G$ one can choose a point $x_n\in U_n$ for all $n\in\N$ in such a way that the resulting sequence $\{x_n:n\in\N\}$ has a $p$-limit in $G$; that is, $\{n\in\N: x_n\in V\}\in p$ for every 
neighbourhood $V$ of $x$ in $G$.
In particular, $G$ is selectively pseudocompact (strongly pseudocompact) 
but
not selectively sequentially pseudocompact. This answers a question of Dorantes-Aldama and the first author.
As a by-product, we show that the free precompact Boolean group over any disjoint sum of maximal countable spaces  
contains no infinite compact subsets.
\end{abstract}

\maketitle

As usual, $\N$ denotes the set of natural numbers, $\omega$ denotes the first infinite cardinal.
We freely identify $\N$ with $\omega$.
We use $\Z_2$ to denote the unique group with two elements.

\section{Results}

Let $p$ be a free ultrafilter on $\N$. Recall  that a point $x$ of a topological space $X$ is a {\em $p$-limit\/} of a sequence 
$\{x_n:n\in\N\}$ of points of $X$ provided that 
$\{n\in\N: x_n\in V\}\in p$ for every neighbourhood $V$ of $x$ in $X$ \cite{Bern}.

The next notion is due to Angoa, Ortiz-Castillo and Tamariz-Mascar\'{u}a
\cite{AOCTM}.

\begin{definition}
Let $p$ be a free ultrafilter on $\N$.
A space $X$ is {\em strongly $p$-pseudocompact\/}
if it has the following property:
For every sequence $\{U_n:n\in\N\}$ of non-empty open subsets of $X$ one can choose a point $x_n\in U_n$ for all $n\in\N$ in such a way that the resulting sequence $\{x_n:n\in\N\}$ has a $p$-limit in $X$. 
\end{definition}

The symbol $\beta\N$ denotes the Stone-\v Cech compactification of $\N$. Recall that $\beta\N\setminus\N$
can be identified with the set of all free ultrafilters on $\N$.

Given a non-empty subset $P$ of $\beta\N\setminus\N$, we shall say that a space $X$ is {\em strongly $P$-pseudocompact\/}
provided that $X$ is strongly $p$-pseudocompact for each $p\in P$.

\begin{theorem}
\label{main:theorem}
Let $\kappa$ be an infinite cardinal such that $\kappa^\omega=\kappa$
and $P$ be a non-empty subset of $\beta\N\setminus\N$ satisfying
$|P|\le\kappa$.
Then there exists a dense 
strongly
$P$-pseudocompact subgroup 
of $\Z_2^\kappa$
without non-trivial convergent sequences.
\end{theorem}

Let $\cont$ denote the cardinality of the continuum.
Applying Theorem \ref{main:theorem} to $P=\beta \N\setminus\N$ and $\kappa=2^\cont$, we get the following
\begin{corollary}
There exists a dense 
strongly
$(\beta \N\setminus\N)$-pseudocompact subgroup 
of $\Z_2^{2^\cont}$
without non-trivial convergent sequences.
\end{corollary}

Given a free ultrafilter $p$ on $\N$, we can 
apply Theorem \ref{main:theorem} to $P=\{p\}$ and $\kappa=\cont$ to get the following
\begin{corollary}
\label{cor:1.4}
For every free ultrafilter $p$ on $\N$, 
there exists a dense 
strongly
$p$-pseudocompact subgroup 
of $\Z_2^{\cont}$
without non-trivial convergent sequences.
\end{corollary}

Garcia-Ferreira and Ortiz-Castillo say that a space $X$ is {\em strongly pseudocompact\/} provided that, for every sequence 
$\{U_n:n\in\N\}$ of non-empty open subsets of $X$, one can choose a point $x_n\in U_n$ for all $n\in\N$ in such a way that the resulting sequence $\{x_n:n\in\N\}$ has a $p$-limit in $X$ for some free ultrafilter on $\N$ (depending on the sequence $\{U_n:n\in\N\}$ in question)
\cite{GF-OC}.
Dorantes-Aldama and Shakhmatov gave a list of equivalent
descriptions of this property in \cite[Theorem 2.1]{DA-S} and proposed an alternative name for it, calling a space $X$ with this property {\em selectively pseudocompact\/} \cite[Definition 2.2]{DA-S}. 
Clearly,
strongly
$p$-pseudocompact spaces are strongly pseudocompact (selectively pseudocompact). 
In particular, 
the groups
from Theorem \ref{main:theorem}
and both of its corollaries above are strongly pseudocompact (selectively pseudocompact). 

Dorantes-Aldama and the first listed author call a space $X$ {\em selectively sequentially pseudocompact\/} provided that, for every sequence 
$\{U_n:n\in\N\}$ of non-empty open subsets of $X$, one can choose a point $x_n\in U_n$ for all $n\in\N$ in such a way that the resulting sequence $\{x_n:n\in\N\}$ has a convergent subsequence
\cite[Definition 2.3]{DA-S}.
Selectively sequentially pseudocompact spaces are strongly pseudocompact (selectively pseudocompact), while the converse does not hold in general \cite{DA-S}.

Since infinite selectively sequentially pseudocompact spaces contain non-trivial convergent sequences \cite[Proposition 3.1]{DA-S}, from Corollary \ref{cor:1.4} and remarks above we obtain the following result:
 
\begin{corollary}
There exists a dense subgroup of $\Z_2^\cont$ which is selectively pseudocompact (strongly pseudocompact) but not selectively sequentially pseudocompact.
\end{corollary}

This corollary provides a positive answer to 
\cite[Question 8.3 (i)]{DA-S}. Consistent examples of such topological 
groups were mentioned already in \cite[Example 5.7]{DA-S}.

We finish this section with the following question.
\begin{question}
Is there an abelian (or even Boolean) group without infinite compact subsets having one of the following properties (listed in an increasing strength):
\begin{itemize}
\item[(i)] selectively pseudocompact (strongly pseudocompact);
\item[(ii)] strongly $p$-pseudocompact for some free ultrafilter $p$;
\item[(iii)] strongly $(\beta\N\setminus\N)$-pseudocompact.
\end{itemize} 
\end{question}

\section{Coherent splitting maps}

Let $X$ be a non-empty set. The set $B(X)=[X]^{<\omega}$ of all finite subsets of $X$ becomes an abelian group with the symmetric difference $E+F=(E\setminus F)\cup(F\setminus E)$ as its group operation $+$ and the empty set as its zero element.
Each element $E$ of $B(X)$ has order $2$, as $E+E=0$.
A
group with this property is often called a {\em Boolean\/} group. Every Boolean group is abelian.

If one abuses notation by identifying an element $x\in X$ with the singleton $\{x\}\in B(X)$, then each element $E\in B(X)$ 
of the group $B(X)$ admits a unique decomposition
$E=\sum_{x\in E} x$, so the set $X$ can be naturally considered as
the set of generators of $B(X)$. 

Every map $f:X\to \Z_2$ has a unique extension $\tilde{f}:B(X)\to \Z_2$ to a homomorphism of $B(X)$ to $\Z_2$ defined 
by $\tilde{f}(E)=\sum_{x\in E} f(x)$ for $E\in B(X)$, where the sum is taken in the group $\Z_2$.
Since the variety $\mathcal{A}_2$ of all Boolean groups is generated by the single group $\Z_2$, the group $B(X)$ coincides with the free group in the variety $\mathcal{A}_2$ over a set $X$
\cite{DS}. Thus, $B(X)$ is the free Boolean group over $X$. 

\begin{definition}
\label{def:splitting}
Let $X$ be a non-empty set.
We shall say that a map $f:X\to \Z_2$  {\em splits\/} a subset $A$ of $B(X)$ provided that 
the set $\{a\in A: \tilde{f}(a)=i\}$ is infinite for each $i=0,1$. 
\end{definition}

Clearly, a subset 
split by some map 
must be infinite.
The converse also holds:

\begin{lemma}
\label{general:splitting:lemma}
Every infinite subset 
of $B(X)$
can be split by some 
map
$f:X\to \Z_2$.
\end{lemma}

This lemma is part of folklore and can be proved by a straightforward induction. It can also be derived from 
\cite[Lemma 4.1]{TY}.

From now on we shall specify the structure of the set $X$.

\begin{definition}
\label{def:coherent}
Let $P$ be a non-empty subset of $\beta\N\setminus\N$
and let $K$ be a non-empty set.
Consider a set of the form $X = P \times K \times (\omega+1)$. 
We shall say that a map $f:X\to \Z_2$ is {\em coherent\/}
provided that
\begin{equation}
\label{eq:coherent}
\{n\in \omega: f(p,\alpha,n)=f(p,\alpha,\omega)\}\in p
\text{ for every }
p\in P
\text{ and each }
\alpha\in K.
\end{equation}
\end{definition}

Note that the map $f:X\to \Z_2$ is coherent if and only if 
$f(p,\alpha,\omega)$ is a $p$-limit of the sequence
$\{f(p,\alpha,n):n\in\N\}$ whenever $p\in P$ and $\alpha\in K$.

\begin{definition}
\label{def:X*}
For a set $X$ as in Definition \ref{def:coherent},
we define
$X^*=P\times K\times \{\omega\}$.
\end{definition}

\begin{lemma}
\label{no:top}
Every map $g:X\setminus X^*\to \Z_2$ admits a unique coherent 
extension  $f:X \to \Z_2$ over $X$.
\end{lemma}
\begin{proof}
For fixed $p\in P$ and $\alpha \in K$,
we have 
$$
\{n\in\omega: g(p,\alpha,n )=0\}
\cup
\{n\in\omega: g(p,\alpha,n)=1\}
=\omega\in p.
$$
Since $p$ is an ultrafilter on $\omega$, there exists 
a unique $i_{\alpha,p}=0,1$ such that
\begin{equation}
\label{eq:3:g}
\{n\in\omega: g(p,\alpha,n)=i_{\alpha,p}\}\in p.
\end{equation}

Define $f(p,\alpha,\omega)=i_{\alpha,p}$ for every $p\in P$ and $\alpha \in K$. 
Finally, let $f(p,\alpha,n)=g(p,\alpha,n)$ for all $(p,\alpha,n)\in X\setminus X^*=P\times K\times \omega$.
It follows from this definition and 
\eqref{eq:3:g} that \eqref{eq:coherent} holds;
that is, $f$ is coherent by Definition \ref{def:coherent}.
\end{proof}

The main goal of this section is to prove the following
theorem strengthening Lemma \ref{general:splitting:lemma} by additionally requiring the splitting map to be coherent. 
\begin{theorem}
\label{coherent:split}
If $X$ is the set as in Definition \ref{def:coherent},
then every infinite subset 
of $B(X)$ can be split 
by some coherent 
map $f:X\to \Z_2$. 
\end{theorem}

The interpretation of this theorem in terms of free precompact Boolean groups shall be given in Section \ref{sec:applications}; see 
Theorem \ref{sequences:in:free:precompact:Boolean:groups}.

The rest of this section is devoted to the proof of this theorem.
It
shall be split into 
two
lemmas dealing with particular cases. 

\begin{lemma}
\label{2.7}
Let $X$ be the set as in Definition \ref{def:coherent}
and $X^*$ be as in Definition \ref{def:X*}.
If $A$ is an infinite subset of $B(X)$ such that
$X^*\cap(\bigcup A)$
is finite, then 
some
coherent 
map $f:X\to \Z_2$ 
splits $A$.
\end{lemma}

\begin{proof}
Since $J= X^*\cap(\bigcup A)$ is finite and $A$ is infinite, there exists $I\in[J]^{<\omega}$ such that the set 
\begin{equation}
\label{eq:A'}
A'=\{a\in A: a\cap X^*=I\}
\end{equation}
is infinite. 
Then 
\begin{equation}
\label{eq:B}
B=\{a\setminus X^*: a\in A'\}=\{a\setminus I: a\in A'\}
\end{equation}
is an infinite subset of $B(X\setminus X^*)$.
By Lemma \ref{general:splitting:lemma}, there exists a
map $g:X\setminus X^*\to\Z_2$
which splits $B$.
Let $f:X\to\Z_2$ be the unique coherent map extending $g$ given by Lemma \ref{no:top}.
Since $g$ splits $B$ and $f$ extends $g$, the map $f$ splits $B$ as well.
It follows from this, \eqref{eq:B} and Definition \ref{def:splitting} that
\begin{equation}
\label{eq:5}
\{a\in A': \tilde{f}(a\setminus I)=i\}
\text{ is infinite for every }
i=0,1.
\end{equation}

Define $j=\tilde{f}(I)$. Clearly, $j\in\{0,1\}$.
It follows from \eqref{eq:A'} and \eqref{eq:B} that
$a = (a\setminus I)\cup I$
for every $a\in A'$, 
so $a=(a\setminus I)+I$ holds in $B(X)$, and therefore,
\begin{equation}
\label{eq:6}
\tilde{f}(a)=\tilde{f}(a\setminus I)+\tilde{f}(I)
=
\tilde{f}(a\setminus I)+j
\ 
\text{ for }
a\in A',
\end{equation}
as $\tilde{f}$ is a homomorphism.
Combining \eqref{eq:5} and \eqref{eq:6}, we conclude that
$\{a\in A': \tilde{f}(a)=i\}$ is infinite for every $i=0,1$.
Since $A'\subseteq A$, the same conclusion holds when $A'$ is replaced by $A$.
According to Definition \ref{def:splitting}, this means that $f$ splits $A$.
\end{proof}

Let $(\Q,\le)$ be a poset.
Recall that a set $D\subseteq \Q$ is said to be
   {\em dense in $(\Q,\le)$\/} provided that for every $r\in \Q$ there exists $q\in D$ such that $q\le r$.

We shall need the following folklore lemma.

\begin{lemma}
\label{countable:generic:filter}
If $\mathscr{D}$ is an at most countable family of dense subsets of a non-empty poset 
$(\Q,\le)$, then there exists an at most countable subset $\F$ of $\Q$
such that $(\F,\le)$ is a linearly ordered set and 
$\F\cap D\not=\emptyset$ for every $D\in\mathscr{D}$.
\end{lemma}
\begin{proof}
Since the family $\mathscr{D}$ is at most countable, we can fix an enumeration $\mathscr{D}=\{D_n:n\in\omega\setminus\{0\}\}$ of elements of $\mathscr{D}$. Since $\Q\not=\emptyset$, there exists $q_0\in\Q$. By induction  on $n\in\omega\setminus\{0\}$, we can choose 
$q_n\in D_n$ such that $q_n\le q_{n-1}$; this is possible because 
$D_n$ is dense in $(\Q,\le)$. Now $\F=\{q_n:n\in\omega\setminus\{0\}\}$ is the desired subset of $\Q$.
\end{proof}

\begin{lemma}
\label{2.9}
Suppose that $P$ and $K$ are at most countable non-empty sets such that $P\subseteq \beta\N\setminus\N$,
$X$ is the set as in Definition \ref{def:coherent}
and $X^*$ is as in Definition \ref{def:X*}.
Assume also that 
$A$ is a subset of $B(X)$ such that
the set
$X^*\cap(\bigcup A)$
is infinite.
Then 
some
coherent 
map $f:X\to \Z_2$
splits $A$.
\end{lemma}

\begin{proof}
We denote by $\Q$ the set of all structures
$q=\structure{P^q,K^q, f^q}$, where
$P^q\in [P]^{<\omega}$, $K^q\in [K]^{<\omega}$,
and $f^q: P^q\times K^q\times(\omega+1)\to \Z_2$ is a coherent map. 
For $q=\structure{P^q,K^q, f^q}, r=\structure{P^r,K^r, f^r}\in \Q$, we let $q\le r$ provided that
$P^r\subseteq  P^q$, $K^r\subseteq K^q$ and $f^q$ extends $f^r$.
One easily sees that $(\Q,\le)$ is a  poset.
Clearly, $\structure{\emptyset,\emptyset,\emptyset}\in\Q$, so $\Q\not=\emptyset$.

\begin{claim}
\label{claim:0}
(i) For every $p\in P$, the set $C_p=\{q\in\Q: p\in P^q\}$ is dense in $(\Q,\le)$.

(ii) For every $k\in K$, the set $E_k=\{q\in\Q: k\in K^q\}$ is dense in $(\Q,\le)$.
\end{claim}
\begin{proof}
(i) Suppose that $r\in\Q$ and $p\in P\setminus P^r$.
Define $P^q=P^r\cup \{p\}$, $K^q=K^r$ and note that the extension
$f^q: P^q\times K^q\times(\omega+1)$ of $f^r$ obtained by letting $f^q(p,k,n)=0$ for all $k\in K^q=K^r$ and $n\in\omega+1$, is coherent, so $q=\structure{P^q, K^q, f^q}\in \Q$. Clearly,
$q\in C_p$ and $q\le r$.

(ii) Suppose that $r\in\Q$ and $k\in K\setminus K^r$.
Define $P^q=P^r$, $K^q=K^r\cup \{k\}$ and note that the extension
$f^q: P^q\times K^q\times(\omega+1)$ of $f^r$ obtained by letting $f^q(p,k,n)=0$ for all $p\in P^q=P^r$ and $n\in\omega+1$, is coherent, so $q=\structure{P^q, K^q, f^q}\in \Q$. Clearly,
$q\in E_k$ and $q\le r$.
\end{proof}

\begin{claim}
\label{claim:2}
For every  $B\in [A]^{<\omega}$ and each $i=0,1$, the set 
\begin{equation}
\label{eq:D{B,i}}
D_{B,i}=\{q\in \Q: \exists a\in A\setminus B\  
(a\subseteq P^q\times K^q\times(\omega+1) \wedge 
\tilde{f^q}(a)=i)\}
\end{equation}
 is dense in $(\Q,\le)$.
\end{claim}
\begin{proof}
Let $r\in \Q$, $B\in[A]^{<\omega}$
 and $i\in\{0,1\}$ be arbitrary.
 We need to find $q\in\Q$ such that $q\le r$,
$a \subseteq P^q\times K^q$ and $\tilde{f^q}(a)=i$.

Since $B$ is finite, the intersection $X^*\cap (\bigcup B)$ is also finite. 
Furthermore,
since both $P^r$ and $K^r$ are finite sets, so is the set 
$P^r\times K^r\times\{\omega\}$.
Therefore, 
\begin{equation}
\label{eq:F}
F=(X^*\cap (\bigcup B))\cup(P^r\times K^r\times\{\omega\})
\end{equation}
  is a finite subset of $X^*$.
By our hypothesis, $X^*\cap(\bigcup A)$ is infinite, so 
there exists $a\in A$ such that $(a\cap X^*)\setminus F\not=\emptyset$.
Fix 
$p_0 \in P$, $k_0 \in K$ and $a \in A$ such that 
$(p_0,k_0,\omega) \in a \setminus F$.
It follows from this and \eqref{eq:F} that $a\in A\setminus B$.

Since $a$ is a finite subset of $X=P\times K\times (\omega+1)$,
there exist finite sets $P^q\subseteq P$ and $K^q\subseteq K$
such that $a\subseteq P^q\times K^q\times (\omega+1)$.
Without loss of generality, we may also assume that
$P^r\subseteq P^q$ and $K^r\subseteq K^q$. 

Let $a'=a\cap (P^r\times K^r\times (\omega+1))$.
Then $j=\tilde{f^r}(a')\in\Z_2$ is well-defined.
There exists a unique $l\in\Z_2$ such that $j+l=i$.
Note that $
(p_0,k_0) \in (P^q\times K^q)
\setminus (P^r\times K^r),$
so we can define a function $f^q: P^q\times K^q\times(\omega+1)$ by 
\begin{equation}
\label{eq:f^q}
f^q(p,k,n)=
\left\{\begin{array}{ll}
f^r(p,k,n) & \mbox{if }  (p,k,n)\in P^r\times K^r\times(\omega+1)\\
l & \mbox{if } (p,k)=(p_0,k_0) \mbox{ and either } n=\omega\mbox{ or } (p,k,n)\not\in a\\
0 & \mbox{otherwise}
\end{array}
\right.
\end{equation}
for all $(p,k,n)\in P^q\times K^q\times(\omega+1)$.

We claim that $q=\structure{P^q,K^q, f^q}\in\Q$.
Since $P^q\in [P]^{<\omega}$ and  $K^q\in [K]^{<\omega}$ by our construction,
we only need to check that the map $f^q: P^q\times K^q\times(\omega+1)\to \Z_2$ is coherent. 
Let $p\in P^q$ and $k\in K^q$ be arbitrary.
If $(p,k)\in P^r\times K^r$, then 
$$
\{n\in \omega: f^q(p,k,n)=f^q(p,k,\omega)\}=\{n\in \omega: f^r(p,k,n)=f^r(p,k,\omega)\}\in p
$$ by \eqref{eq:f^q} and
coherency of $f^r$.
Suppose now that $(p,k)\in (P^q\times K^q)\setminus (P^r\times K^r)$.
If $(p,k)\not=(p_0,k_0)$, then
$f^q(p,k,n)=0$ for all $n\in\omega+1$ by \eqref{eq:f^q}, so
$\{n\in\omega: f^q(p,k,n)=f^q(p,k,\omega)\}=\omega\in p$.
Finally, if $(p,k)=(p_0,k_0)$, then the second line of 
\eqref{eq:f^q} implies that $f^q(p,k,\omega)=l$
and $f^q(p,k,n)=l$ for all but finitely many $n\in\omega$, as the set $a$ is finite. Therefore, $\{n\in \omega: f^q(p,k,n)=f^q(p,k,\omega)\}$ is a cofinite subset of $\omega$, so it belongs to $p$ as $p$ is a free ultrafilter on $\omega$.
This finishes the check of the inclusion $q\in \Q$. Clearly,
$q\le r$.

Let us show that $q\in D_{B,i}$.
The only condition that remains to be checked is the equality $\tilde{f^q}(a)=i$.

First, let us show that $\tilde{f^q}(a\setminus a')=l$.
Indeed, 
since
$a'=a\cap (P^r\times K^r\times (\omega+1))$,
we have 
$a\setminus a'\subseteq (P^q\times K^q)\setminus (P^r\times K^r)$,
so \eqref{eq:f^q} implies that $f^q(p_0,k_0,\omega)=l$
and 
$f^q(p,k,n)=0$ for all $(p,k,n)\in a\setminus a'$ such that 
$(p,k)\not=(p_0,k_0)$.
(Recall that $(p_0,k_0,\omega)\in a\setminus a'$ by our choice.)
Since $\tilde{f^q}$ is a homomorphism, this implies
$$
\tilde{f^q}(a\setminus a')=\sum_{(p,k,n)\in a\setminus a'}\tilde{f^q}(\{p,k,n\})=\sum_{(p,k,n)\in a\setminus a'} f^q(p,k,n)=
f^q(p_0,k_0,\omega)=l.
$$

Second,
note that $a=(a\setminus a')\cup a'$, so $a=(a\setminus a')+a'$.
Since $\tilde{f^q}$ is a homomorphism,
$
\tilde{f^q}(a)=\tilde{f^q}(a\setminus a')+\tilde{f^q}(a')=
l+j=i.
$
\end{proof}

By Claims \ref{claim:0} and \ref{claim:2}, the family
$$\mathscr{D} = 
\{C_p: p\in P\}
\cup
\{E_k:k\in K\}
\cup
\{D_{B,i} : B \in [A]^{< \omega}, i \in \{0,1\} \}
$$
consists of dense subsets of $(\Q,\le)$.
Since $P$, $K$ and $A$ are at most countable, so is $\mathscr{D}$.
By Lemma \ref{countable:generic:filter}, there exists a set $\F=\{q_n:n\in\N\}\subseteq\Q$
such that $q_0\ge q_1\ge \dots\ge q_n\ge q_{n+1}\ge\dots$
and 
$\F\cap D\not=\emptyset$ for every $D\in\mathscr{D}$.
We claim that  $f=\bigcup\{f^{q_n}:n\in\N\}$ is the coherent map from $X$ to $\Z_2$ splitting $A$.
Since $\F$ intersects each $C_p$ and every $E_k$, the domain of $f$ coincides with $X=P\times K\times(\omega+1)$.
Since each $f^{q_n}$ is coherent and $f$ extends all $f^{q_n}$, 
it easily follows that $f$ is coherent as well.

Suppose 
that $f$ does not split $A$.
Then the set $B=\{a\in A: \tilde{f}(a)=i\}$ must be finite for some $i=0,1$, so
$B\in[A]^{<\omega}$ and thus,
$D_{B,i}\in \mathscr{D}$. Therefore, $q_n\in D_{B,i}$ for some $n\in\N$.
Applying \eqref{eq:D{B,i}}, we can find $a\in A\setminus B$
such that $a\subseteq P^{q_n}\times K^{q_n}\times(\omega+1)$
and 
$\tilde{f^{q_n}}(a)=i$. Since $f^{q_n}\subseteq f$, this implies
$\tilde{f}(a)=\tilde{f^{q_n}}(a)=i$.
Therefore, $a\in B$ by the definition of the set $B$, in contradiction with $a\in A\setminus B$.
\end{proof}

\medskip
\noindent
{\bf Proof of Theorem \ref{coherent:split}.}
Let $A$ be an infinite subset of $X$. Choose a countably infinite subset $A'$ of $A$. Since $A'\subseteq B(X)=[X]^{<\omega}$, there exists at most countable sets $P'\subseteq P$ and $K'\subseteq K$ such that $A\subseteq B(X')$, where $X'=P'\times K'\times (\omega+1)$.
Combining Lemmas \ref{2.7}
and
\ref{2.9}, we can find a coherent map
$f':X'\to\Z_2$ splitting $A'$.
Let $f:X\to \Z_2$ be the extension of $f'$ over $X$ obtained by 
letting $f$ to take $0$ everywhere on $X\setminus X'$.
Clearly, $f$ is a coherent map which splits $A'$.
Since $A'\subseteq A$, $f$ splits $A$ as well.
\qed

\section{Proof of Theorem \ref{main:theorem}}

Fix a cardinal
$\kappa$
such that $\kappa^\omega=\kappa$,
and let $P$ be a non-empty subset of $\beta\N\setminus\N$ satisfying
$|P|\le\kappa$.
Define
$$
X=P\times\kappa\times(\omega+1)
\ 
\text{ and } 
\ 
X^*=P\times\kappa\times\{\omega\}.
$$

Note that
$|X|=\kappa$ by our assumption on $\kappa$ and $P$, 
so
$|B(X)|=|X^{<\omega}|=|X|=\kappa$
and
$|[B(X)]^\omega|=\kappa^\omega=\kappa$,
where $[B(X)]^\omega$ denotes the family of all countable subsets of $B(X)$.
Therefore,
we can fix an enumeration
$$
[B(X)]^\omega=\{A_\beta:\beta\in\kappa\}
$$
such that for every $A\in [B(X)]^\omega$ the set
$\{\beta\in\kappa: A_\beta=A\}$ is cofinal in $\kappa$.
For every $\beta\in\kappa$, use Theorem \ref{coherent:split} to  fix a coherent 
map $f_\beta:X\to\Z_2$
splitting $A_\beta$.

Since
$|[\kappa]^\omega|=\kappa$
by our assumption on $\kappa$, 
and $|(\Z_2^I)^\omega|=\cont\le\kappa$ for every $I\in[\kappa]^\omega$, 
we can fix an enumeration $[\kappa]^\omega=\{I_\alpha:\alpha<\kappa\}$
and a sequence $\{y_{p,\alpha,n}:n\in\omega\}\subseteq \Z_2^{I_\alpha}$ for every $(p,\alpha)\in P\times\kappa$
such that whenever
$I\in [\kappa]^\omega$, $p\in P$
and  $\{y_n:n\in\omega\}\subseteq\Z_2^I$, there exists $\alpha\in\kappa$ with
$I_\alpha=I$ and $y_{p,\alpha,n}=y_n$ for all $n\in\omega$.

Let $\alpha\in\kappa$ be arbitrary. For every $p\in P$, 
the sequence $\{y_{p,\alpha,n}:n\in\omega\}$ of points of the compact space $\Z_2^{I_\alpha}$ has a $p$-limit $y_{p,\alpha,\omega}\in \Z_2^{I_\alpha}$.

For each 
$(p,\alpha,n)\in X$,
define
$z_{p,\alpha,n}\in \Z_2^\kappa$ by
\begin{equation}
\label{eq:z}
z_{p,\alpha,n}(\beta)=
\left\{\begin{array}{ll}
y_{p,\alpha,n}(\beta) & \mbox{if } \beta\in I_\alpha \\
f_\beta(p,\alpha,n) & \mbox{if } \beta\in \kappa\setminus I_\alpha
\end{array}
\right.
\hskip30pt
\text{for every }
\beta\in\kappa.
\end{equation}
Finally, let 
\begin{equation}
Z=\{z_{p,\alpha,n}: (p,\alpha,n)\in X\}.
\end{equation}

\begin{claim}
\label{claim:1}
$Z$ is a dense, strongly
$P$-pseudocompact
subset of $\Z_2^\kappa$.
\end{claim}
\begin{proof}
Let $p\in P$ be arbitrary. 
We are going to check that $Z$ is 
strongly
$p$-pseudocompact.
Let $\{U_n:n\in\omega\}$ be a sequence of non-empty basic open subsets of $\Z_2^\kappa$. Then for every $n\in\omega$,
we have $U_n=\prod_{\beta\in\kappa} U_{\beta,n}$, where each 
$U_{\beta,n}$ is a non-empty open subset of $\Z_2$ and 
$\supp(U_n)=\{\beta\in\kappa: U_{\beta,n}\neq \Z_2\}$
is a finite subset of $\kappa$.
Then the set $J=\bigcup_{n\in\omega} \supp(U_n)$
is at most countable, so we can fix a countably infinite subset $I$ of $\kappa$ containing $J$.

For every $n\in\omega$,
$V_n=\prod_{\beta\in I} U_{\beta,n}$ is a non-empty subset of $\Z_2^I$, so we can select $y_n\in V_n$.

By the property of our enumeration, there exists $\alpha\in\kappa$
such that $I_\alpha=I$ and $y_{p,\alpha,n}=y_n$ for all $n\in\omega$.
By \eqref{eq:z}, for every $n\in\omega$, we have
\begin{equation}
\label{eq:16}
z_{p,\alpha,n}(\beta)=y_{p,\alpha,n}(\beta)=y_n(\beta)\in U_{\beta,n}
\text{ for every }\beta\in I_\alpha=I.
\end{equation}
Since $\supp(U_n)\subseteq I$, this implies $z_{p,\alpha,n}\in U_n$ for each $n\in\omega$.
Since $\{z_{p,\alpha,n}:n\in\omega+1\}\subseteq Z$, it suffices to check that $z_{p,\alpha,\omega}$ is a $p$-limit of the sequence $\{z_{p,\alpha,n}:n\in\omega\}$.
To achieve this, one only needs to show that 
$z_{p,\alpha,\omega}(\beta)$ is a $p$-limit of the sequence $\{z_{p,\alpha,n}(\beta):n\in\omega\}$ for every $\beta\in\kappa$.
We consider two cases.

\smallskip
{\em Case 1\/}. $\beta\in I_\alpha$.
Since
the sequence $\{y_{p,\alpha,n}:n\in\omega\}\subseteq\Z_2^{I_\alpha}$ has a $p$-limit $y_{p,\alpha,\omega}\in \Z_2^{I_\alpha}$, it follows that
$y_{p,\alpha,\omega}(\beta)$ is a $p$-limit of the sequence
$\{y_{p,\alpha,n}(\beta):n\in\omega\}$.
Since $\beta\in I_\alpha$, we have 
$z_{p,\alpha,\omega}(\beta)=y_{p,\alpha,\omega}(\beta)$.
Combining this with \eqref{eq:16}, we get the desired conclusion.

\smallskip
{\em Case 2\/}. $\beta\in \kappa\setminus I_\alpha$.
In this case, it follows from \eqref{eq:z} that
$z_{p,\alpha,n}(\beta)=f_\beta(p,\alpha,n)$
for every $n\in\omega+1$,
and the conclusion follows from the fact that $f_\beta$ is coherent. 

To prove that $Z$ is dense in $\Z_2^\kappa$, consider an arbitrary non-empty open subset $U$ of $\Z_2^\kappa$. Applying the above argument to a fixed $p\in P$ and the sequence $\{U_n:n\in\omega\}$, where $U_n=U$ for every $n\in\omega$, we find $\alpha\in\kappa$ such that
$\{z_{p,\alpha,n}:n\in\omega\}\subseteq U$; in particular,
$Z\cap U\not=\emptyset$.
\end{proof}

\begin{claim}
The subgroup $G$ of $\Z_2^\kappa$ generated by $Z$
contains no non-trivial convergent sequences.
\end{claim}
\begin{proof}
Consider an arbitrary faithfully indexed sequence $\{g_m:m\in\omega\}\subseteq G$.
Since $G$ is a Boolean group generated by the set $Z$, for every $m\in\omega$ there exists a finite subset $E_m$ of $X$ such that
\begin{equation}
\label{eq:g_m}
g_m=\sum_{(p,\alpha,n)\in E_m} z_{p,\alpha,n}.
\end{equation}
Since the sequence $\{g_m:m\in\omega\}$ is faithfully indexed,  so is the sequence $A=\{E_m:m\in\omega\}\subseteq [X]^{<\omega}=B(X)$. In particular, $A\in [B(X)]^\omega$.
By the choice of our enumeration, 
the set
$B=\{\beta\in\kappa: A_\beta=A\}$ is cofinal in $\kappa$.

Since $E=\bigcup\{E_m:m\in\omega\}$ is at most countable subset of $X$, the set 
\begin{equation}
\label{eq:J}
J=\{\alpha\in\kappa: \exists p\in P\ \exists\ n\in(\omega+1)\ (p,\alpha,n)\in E\}
\end{equation}
 is at most countable as well.
Therefore, 
\begin{equation}
\label{eq:I}
I=\bigcup_{\alpha\in J}I_\alpha
\end{equation}
  is an at most countable subset of $\kappa$. Since $\kappa\ge\cont$ and $B$ is cofinal in $\kappa$,
we can find $\beta\in B\setminus I$.

Let $m\in\omega$ be arbitrary.
Suppose that $(p,\alpha,n)\in E_m$.
Then $\alpha\in J$ by the inclusion $E_m\subseteq E$ and \eqref{eq:J}.
Therefore, $I_\alpha\subseteq I$ by \eqref{eq:I}.
Since $\beta\not\in I$, we conclude that $\beta\in\kappa\setminus I_\alpha$, and thus
$z_{p,\alpha,n}(\beta)=f_\beta(p,\alpha,n)$ by \eqref{eq:z}.
Since this holds for every $(p,\alpha,n)\in E_m$ and $\tilde{f_\beta}$ is a homomorphism, we conclude that
$$
\tilde{f_\beta}(E_m)=\tilde{f_\beta}\left(\sum_{(p,\alpha,n)\in E_m} \{(p,\alpha,n)\}\right)
=
\sum_{(p,\alpha,n)\in E_m}
\tilde{f_\beta}(\{(p,\alpha,n)\})
=
$$
$$
=
\sum_{(p,\alpha,n)\in E_m}
f_\beta(p,\alpha,n)=
\sum_{(p,\alpha,n)\in E_m}
z_{p,\alpha,n}(\beta)
=g_m(\beta),
$$
where the last equality is due to 
\eqref{eq:g_m}.
Since $f_\beta$ splits the sequence $A=\{E_m:m\in\omega\}$,
this means that the set $\{m\in\omega: g_m(\beta)=i\}$ is infinite for both
$i=0,1$. This implies that $\{g_m:m\in\omega\}$ cannot be a convergent sequence in $G$.
\end{proof}

Since $Z\subseteq G\subseteq \Z_2^\kappa$ and $Z$ is dense in $\Z_2^\kappa$ by Claim \ref{claim:1},
$Z$ is dense in $G$.
Since $Z$ is 
strongly
$P$-pseudocompact by Claim \ref{claim:1},
so is $G$.
\qed

\section{An application to free precompact Boolean groups}
\label{sec:applications}

In this section we give an interpretation of Theorem \ref{coherent:split} in terms of free precompact Boolean groups
over disjoint sums of countable maximal spaces.

Recall that a topological group is {\em precompact\/} if it is a subgroup of some compact group, or equivalently, if its completion is compact. The class of all precompact Boolean groups forms a variety $\mathscr{V}$ of topological groups
\cite{Hi, Mo}.
Therefore,
given a topological space $X$, there exists the free object of $X$ 
in $\mathscr{V}$ \cite{Mo-survey, CvM} which we shall call the
{\em free precompact Boolean group\/} $\mathit{FPB}(X)$ of $X$.

A description of this object as the reflection of the free (Abelian) topological group of a space $X$ in the class of precompact Boolean groups can be found in \cite[Section 9]{SS}.
Here is another description.

\begin{lemma}
\label{desscription:of:topology}
Let $X$ be a zero-dimensional topological space and let $\mathscr{F}$ be the family of all continuous maps $f:X\to\Z_2$ from $X$ to the group $\Z_2$ endowed with the discrete topology. Consider the initial topology $\mathscr{T}$ on $B(X)$ with respect to the family $\tilde{\mathscr{F}}=\{\tilde{f}: f\in\mathscr{F}\}$ of homomorphisms; that is, the family, $\{\tilde{f}^{-1}(i): f\in\mathscr{F}, i=0,1\}$ forms a subbasis for the topology $\mathscr{T}$.
Then $\mathit{FPB}(X)$ is topologically isomorphic to $(B(X),\mathscr{T})$ and $\mathscr{T}$ induces the original topology on $X$ after identification of $X$ with the subspace $\{\{x\}:x\in X\}$ of $B(X)$.
\end{lemma}
\begin{proof}
Let us first notice four facts.
\begin{itemize}
\item[(i)] The completion $K$ of a precompact Boolean group is a compact Boolean group, and the standard facts of the duality theory imply that $K$ is topologically isomorphic to the Cartesian product $\Z_2^\tau$ for some cardinal $\tau$.
\item[(ii)] It follows from (i) that the variety $\mathscr{V}$ of precompact Boolean groups is generated by the single topological group $\Z_2$ (taken with the discrete topology).
\item[(iii)] Since $X$ is zero-dimensional, the family $\mathscr{F}$ separates points and closed subsets of $X$, so the diagonal map 
$j$ of the family $\mathscr{F}$ is a homeomorphic embedding of $X$ 
into $\Z_2^\mathscr{F}$.
\item[(iv)] Since $X$ is zero-dimensional, $j(X)$ is $\mathcal{A}_2$-independent subset of $\Z_2^\mathscr{F}$, where $\mathcal{A}_2$ is the algebraic variety of groups of order $2$; see \cite[Section 5]{DS}.
\end{itemize}

The conclusion of our lemma now 
follows from these facts via a standard argument.
\end{proof}

We refer the reader to \cite[Section 2]{Sh} for properties of free precompact (Abelian) groups and \cite{Si} for those of free precompact Boolean groups.

For simplicity, we shall say that a space $E$ is {\em elementary\/}
if it is homeomorphic to a subspace of the form $\N\cup\{p\}$ of $\beta\N$, where $p\in\beta\N\setminus\N$. Alternatively, 
an elementary space is a countably infinite space $E$ with a single non-isolated point $p$ such that the trace of the filter of neighbourhoods of $p$ on the discrete part  $D=E\setminus \{p\}$ 
of $E$ is an ultrafilter on $D$.

Recall that a space is {\em maximal\/} if it is non-discrete yet any strictly stronger topology on it is discrete. 
One can easily see that elementary spaces are precisely 
the countably infinite maximal spaces.

The alternative reformulation of Theorem \ref{coherent:split} is the following

\begin{theorem}
\label{sequences:in:free:precompact:Boolean:groups}
The free precompact Boolean group $\mathit{FPB}(\bigoplus_{i\in I} E_i)$
of any
topological sum $\bigoplus_{i\in I} E_i$ of elementary spaces $E_i$ ($i\in I$) contains no non-trivial convergent sequences.
\end{theorem}
\begin{proof}
First, we prove our theorem in the special case.

Let $X$ be the set as in Definition \ref{def:coherent}
and $X^*$ be as in Definition \ref{def:X*}. Introduce the following topology on $X$. Declare each point of $X\setminus X^*$ to be isolated. A basic open neighbourhood of a point $(p,k,\omega)\in X^*$ is of the form $\{(p,k,\omega)\}\cup\{(p,k,n):n\in F\}$ for a given element $F\in p$. Note that each 
$E_{p,k}=\{(p,k,n):n\in\omega+1\}$ for $(p,k)\in P\times K$ 
is a clopen subset of $X$ homeomorphic to the elementary space 
$\N\cup\{p\}$, so 
$X=\bigoplus_{(p,k)\in P\times K} E_{p,k}$
is a topological sum of elementary spaces
$E_{p,k}$.

It is straightforward to check that, when $X$ is equipped with the topology described above, a map $f:X\to\Z_2$ is continuous if and only if it is coherent in the sense of Definition
\ref{def:coherent}.
Therefore, the family $\mathscr{F}$
in Lemma \ref{desscription:of:topology} is precisely the family of all 
coherent maps from $X$ to $\Z_2$.

Now, let $A$ be an infinite subset of $\mathit{FPB}(X)$.
By Lemma \ref{desscription:of:topology}, we can identify $\mathit{FPB}(X)$ with $(B(X),\mathscr{T})$, where $\mathscr{T}$ is the topology on $B(X)$ described in this lemma.
After this identification, we can think of $A$ as being an infinite subset of $B(X)$. By Theorem \ref{coherent:split},
$A$ is split by some coherent map $f:X\to\Z_2$. By the previous paragraph, $f\in\mathscr{F}$.
Therefore, 
$B(X)=U_0\cup U_1$
is a $\mathscr{T}$-clopen partition of $B(X)$,
where
$U_i=\tilde{f}^{-1}(i)$ for $i=0,1$.
Since $A$ is split by $f$, the intersection $A\cap U_i$ is infinite for both $i=0,1$. This shows that $A$ cannot be a convergent sequence in $(B(X),\mathscr{T})$.

Having this particular case proved, consider now a general case of a topological 
sum $E=\bigoplus_{i\in I} E_i$ of elementary spaces $E_i$.
For $i\in I$, let $E_i=\{(i,n):n\in\omega+1\}$ where each point 
$(i,n)$ for $n\in\omega$ is isolated and the point $(i,\omega)$ is non-isolated; let $p_i\in\beta\N\setminus\N$ be the filter on $\N$ obtained by ``copying over $\N$'' the trace on $\{(i,n):n\in\omega\}$ of the filter of neighbourhoods of  the non-isolated point $(i,\omega)$ in $E_i$.

In the special case, take $P=\beta\N\setminus\N$ and $K=I$;
that is, consider the topological sum
$$
X=\bigoplus_{(p,i)\in(\beta\N\setminus\N)\times I} E_{p,i}.
$$
Let $h:E\to X$ be the natural homeomorphic embedding which sends each point $(i,n)\in E_i$ to 
$(p_i,i,n)\in E_{p_i,i}$ for $i\in I$ and $n\in\omega+1$.
Since every bounded continuous map $g:h(E)\to\Z_2$ admits a continuous extension $\hat{g}:X\to\Z_2$ over $X$,
it easily follows from Lemma \ref{desscription:of:topology}
that $\mathit{FPB}(E)$ is a subgroup of $\mathit{FPB}(X)$.
Since we have already proved that the bigger group $\mathit{FPB}(X)$ contains no non-trivial convergent sequences, the same conclusion holds for $\mathit{FPB}(E)$.
\end{proof}

It should be noted that Theorem \ref{sequences:in:free:precompact:Boolean:groups}
is not completely trivial, as the completion of any precompact
group is a compact group, and compact groups contain many non-trivial convergent sequences.

In fact, even a stronger conclusion holds in Theorem \ref{sequences:in:free:precompact:Boolean:groups}.
\begin{corollary}
The free precompact Boolean group $G=\mathit{FPB}(\bigoplus_{i\in I} E_i)$
of any
topological sum $\bigoplus_{i\in I} E_i$ of elementary spaces $E_i$ ($i\in I$) contains no infinite compact subsets.
\end{corollary}
\begin{proof}
Suppose that $K$ is an infinite compact subset of $G$.
Choose a countably infinite subset $S$ of $K$. Then its closure in $G$ is compact as well, so without loss of generality we may assume that $S$ is dense in $K$. Since $S$ is countable,
there exists an at most countable set $J\subseteq I$ such that
$S\subseteq H=\mathit{FPB}(\bigoplus_{i\in J} E_i)$.
Since $\bigoplus_{i\in J} E_i$ is clopen in $\bigoplus_{i\in I} E_i$,
$H$ is a closed subgroup of $G$. Since $S\subseteq H$ and $S$ is dense in $K$, it follows that $K\subseteq H$.
Since the group $H$ is countable, so is $K$. Therefore, the compact space $K$ must be metrizable. As it is infinite, it must contain a non-trivial convergent sequence. We have found a non-trivial convergent sequence in $G$, in contradiction with Theorem 
\ref{sequences:in:free:precompact:Boolean:groups}.
This contradiction shows that all compact subsets of $G$ are finite.
\end{proof}

\end{document}